\documentclass[a4paper,12pt,final]{amsart}
\usepackage{times,a4wide,mathrsfs,amssymb,amsmath,amsthm}

\newcommand{\C}{\mathbb{C}}
\newcommand{\ZZ}{\mathbb{Z}}

\newcommand{\LLL}{\mathbb{L}}
\newcommand{\QQ}{\mathbb{Q}}

\newcommand{\PP}{\mathbb{P}}

\newcommand{\OO}{\mathcal O}
\newcommand{\Ss}{\mathcal S}

\newcommand{\XX}{\mathcal X}

\newcommand{\VV}{\mathcal V}
\newcommand{\WW}{\mathcal W}
\newcommand{\FF}{\mathcal F}

\newcommand{\MM}{\mathcal M}
\newcommand{\Pp}{\mathcal P}

\newcommand{\pic}{\hbox{Pic}}

\newcommand{\codim}{\hbox{codim}}

\newcommand{\gr}{\hbox{Gr}}
\newcommand{\wt}{\widetilde}
\newcommand{\rom}{\romannumeral}

\DeclareMathOperator{\aut}{Aut}
\DeclareMathOperator{\ide}{id}

\DeclareMathOperator{\ima}{Im}

\newtheorem{theorem}{Theorem}[section]

\newtheorem{lemma}[theorem]{Lemma}

\newtheorem{corollary}[theorem]{Corollary}
\newtheorem{proposition}[theorem]{Proposition}

\newtheorem{remark}[theorem]{Remark}
\newtheorem{definition}[theorem]{Definition}
\newtheorem{convention}{Conventions}

\newtheorem{nonumbering}{Theorem}

\newtheorem{nonumberingc}{Corollary}

\newtheorem{nonumberingt}{Acknowledgements}

\begin{document}
\author[Robert Laterveer]
{Robert Laterveer}

\address{Institut de Recherche Math\'ematique Avanc\'ee,
CNRS -- Universit\'e 
de Strasbourg,\
7 Rue Ren\'e Des\-car\-tes, 67084 Strasbourg CEDEX,
FRANCE.}
\email{robert.laterveer@math.unistra.fr}

\title{On the Chow groups of certain cubic fourfolds}

\begin{abstract} This note is about the Chow groups of a certain family of smooth cubic fourfolds. This family is characterized by the property that each cubic fourfold $X$ in the family has an involution such that the induced involution on the Fano variety $F$ of lines in $X$ is symplectic and has a $K3$ surface $S$ in the fixed locus. The main result establishes a relation between $X$ and $S$ on the level of Chow motives. As a consequence, we can prove finite--dimensionality of the motive of certain members of the family. 
\end{abstract}


\maketitle

\section{Introduction}

For a smooth projective variety $X$ over $\C$, let $A^i(X):=CH^i(X)_{\QQ}$ denote the Chow groups (i.e. the groups of codimension $i$ algebraic cycles on $X$ with $\QQ$--coefficients, modulo rational equivalence). Let $A^i_{hom}(X)$ denote the subgroup of homologically trivial cycles. 

When $X\subset\PP^5(\C)$ is a smooth cubic fourfold, we have $A^i_{hom}(X)=0$ for $i\not=3$, but $A^3_{hom}(X)\not=0$ (this is related to the fact that $H^{3,1}(X)\not=0$).
The main result of this note shows that for a certain family of cubic fourfolds, the group $A^3_{hom}(X)$ is not larger than the Chow group of $0$--cycles on a $K3$ surface:

\begin{nonumbering}[=theorem \ref{main}] Let $X\subset\PP^5(\C)$ be a smooth cubic fourfold defined by an equation
  \[  f(x_0,x_1,x_2,x_3) + (x_4)^2\ell_1(x_0,\ldots,x_3) + (x_5)^2\ell_2(x_0,\ldots,x_3)+x_4x_5\ell_3(x_0,\ldots,x_3)=0\ \]
  (here $f$ has degree $3$ and $\ell_1, \ell_2,\ell_3$ are linear forms). There exists a $K3$ surface $S$ and a correspondence $\Gamma\in A^{3}(X\times S)$ inducing a split injection
   \[ \Gamma_\ast\colon\ \ \ A^3_{hom}(X)\ \hookrightarrow\ A^2_{hom}(S)\ .\]
 \end{nonumbering}
  
  In a nutshell, the argument proving theorem \ref{main} is as follows: cubics $X$ as in theorem \ref{main} have an involution inducing a symplectic involution $\iota_F$ of the Fano variety of lines $F=F(X)$. The fixed locus of $\iota_F$ contains a $K3$ surface $S$. The inclusion $S\subset F$ being symplectic, there is a (correspondence--induced) isomorphism
   \[  \Gamma_\ast\colon\ \ \ H^{3,1}(X)\ \xrightarrow{\cong}\ H^{2,0}(S)\ .\]
   Because the cubics $X$ as in theorem \ref{main} form a large family, and the correspondence $\Gamma$ exists for the whole family, one can apply Voisin's method of ``spread'' \cite{V0}, \cite{V1}, \cite{Vo}, \cite{Vo2} to this isomorphism, and obtain a statement on the level of rational equivalence which proves theorem \ref{main}.
     
    As an application of theorem \ref{main}, we obtain some new examples of cubics with finite--dimensional motive (in the sense of Kimura/O'Sullivan \cite{Kim}, \cite{An}, \cite{J4}):
    
 \begin{nonumberingc}[=corollary \ref{findim}] Let $X$ be as in theorem \ref{main}, and   
  assume
   \[ \dim H^4(X)\cap H^{2,2}(X,\C)\ge 20 \ .\]
   Then $X$ has finite--dimensional motive.    
   \end{nonumberingc}
   
  For $X$ as in corollary \ref{findim}, one can also prove finite--dimensionality for the Fano varieties of lines on $X$ (remark \ref{fanotoo}). This gives new examples of 
  hyperk\"ahler fourfolds with finite--dimensional motive.

 \vskip0.6cm

\begin{convention} In this article, the word {\sl variety\/} will refer to a reduced irreducible scheme of finite type over $\C$. A {\sl subvariety\/} is a (possibly reducible) reduced subscheme which is equidimensional. 

{\bf All Chow groups will be with rational coefficients}: we will denote by $A_j(X)$ the Chow group of $j$--dimensional cycles on $X$ with $\QQ$--coefficients; for $X$ smooth of dimension $n$ the notations $A_j(X)$ and $A^{n-j}(X)$ are used interchangeably. 

The notations $A^j_{hom}(X)$, $A^j_{AJ}(X)$ will be used to indicate the subgroups of homologically trivial, resp. Abel--Jacobi trivial cycles.
For a morphism $f\colon X\to Y$, we will write $\Gamma_f\in A_\ast(X\times Y)$ for the graph of $f$.
The contravariant category of Chow motives (i.e., pure motives with respect to rational equivalence as in \cite{Sc}, \cite{MNP}) will be denoted $\MM_{\rm rat}$.



We will write $H^j(X)$ 
to indicate singular cohomology $H^j(X,\QQ)$.

Given a group $G\subset\aut(X)$ of automorphisms of $X$, we will write $A^j(X)^G$ (and $H^j(X)^G$) for the subgroup of $A^j(X)$ (resp. $H^j(X)$) invariant under $G$.
\end{convention}

\section{Preliminaries}

\subsection{Refined K\"unneth decomposition}

\begin{definition} Let $X$ be a smooth projective variety, and $h\in\pic(X)$ an ample class.
 The hard Lefschetz theorem asserts that the map
  \[  L^{n-i}\colon H^i(X)\ \to\ H^{2n-i}(X)\]
  obtained by cupping with $h^{n-i}$ is an isomorphism, for any $i< n$. One of the standard conjectures, often denoted $B(X)$, asserts that the inverse isomorphism is algebraic: we say that $B(X)$ holds if for any $i<n$, there exists a correspondence $\C_i\in A^{i}(X\times X)$ such that 
  \[ (C_i)_\ast\colon\ \ H^{2n-i}(X)\ \to\ H^i(X) \]
  is an inverse to $L^{n-i}$.
\end{definition}

\begin{remark} For more on the standard conjectures, cf. \cite{K0}, \cite{K1}. In this note, we will be using the following two facts: Any smooth hypersurface $X\subset\PP^m(\C)$ verifies $B(X)$ \cite{K0}, \cite{K1}. For any smooth cubic fourfold $X\subset\PP^5(\C)$, the Fano variety of lines $F:=F(X)$ verifies $B(F)$ (this follows from \cite[Theorem 1.1]{ChMa}, or alternatively from \cite[Corollary 6]{moi2}).
\end{remark}

\begin{remark} Let $N^\ast H^\ast$ denote the coniveau filtration on cohomology \cite{BO}. Vial \cite{V4} has introduced a variant filtration $\wt{N}^\ast H^\ast$, called the {\em niveau filtration\/}. There is an inclusion
  \[ \wt{N}^j H^i(X)\ \subset\ N^j H^i(X) \]
  for any $X$ and all $i,j$. Conjecturally, this is always an equality (this would follow from the standard conjecture $B$). If $B(X)$ holds and $j\ge {i-1\over 2}$, this inclusion is an equality \cite[]{V4}.
\end{remark}

\begin{theorem}[Vial \cite{V4}]\label{vial} Let $X$ be a smooth projective variety of dimension $n\le 5$. Assume $B(X)$ holds.
There exists a decomposition of the diagonal
  \[ \Delta_X={\displaystyle \sum_{i,j} } \pi^X_{i,j}\ \ \ \hbox{in}\  H^{2n}(X\times X)\ ,\]
  where the $\pi_{i,j}$'s are mutually orthogonal idempotents. The correspondence $\pi_{i,j}$ acts on $H^\ast(X)$ as a projector on $\gr^j_{\wt{N}} H^i(X)$. Moreover, $\pi_{i,j}$ can be chosen to factor over a variety of dimension $i-2j$ (i.e., for each $\pi_{i,j}$ there exists a smooth projective variety $Z_{i,j}$ of dimension $i-2j$, and correspondences $\Gamma_{i,j}\in A^{n-j}(Z_{i,j}\times X), \Psi_{i,j}\in A^{i-j}(X\times Z_{i,j})$ such that $\pi_{i,j}=\Gamma_{i,j}\circ \Psi_{i,j}$ in $H^{2n}(X\times X)$).
\end{theorem}

\begin{proof} This is a special case of \cite[Theorem 1]{V4}. Indeed, as mentioned in loc. cit., varieties $X$ of dimension $\le 5$ such that $B(X)$ holds verify condition (*) of loc. cit. 
\end{proof}

 \begin{remark} If $X$ is a surface, $\pi^X_{2,0}$ is the homological realization of the projector $\pi^X_{2,tr}$ constructed on the level of Chow motives in \cite{KMP}. 
  \end{remark}

 \subsection{Spread}

  \begin{lemma}[Voisin \cite{V0}, \cite{V1}]\label{projbundle} Let $M$ be a smooth projective variety of dimension $n+1$, and $L$ a very ample line bundle on $M$. Let 
    \[ \pi\colon \XX\to B\]
    denote a family of hypersurfaces, where $B\subset\vert L\vert$ is a Zariski open.
      Let
   \[   p\colon \wt{\XX\times_B \XX}\ \to\ \XX\times_B \XX\]
   denote the blow--up of the relative diagonal. 
 Then $\wt{\XX\times_B \XX}$ is Zariski open in $V$, where $V$ is a projective bundle over $\wt{M\times M}$, the blow--up of $M\times M$ along the diagonal.
   \end{lemma} 
  
  \begin{proof} This is \cite[Proof of Proposition 3.13]{V0} or \cite[Lemma 1.3]{V1}. The idea is to define $V$ as
   \[  V:=\Bigl\{ \bigl((x,y,z),\sigma\bigr) \ \vert\ \sigma\vert_z=0\Bigr\}\ \ \subset\ \wt{M\times M}\times \vert L\vert\ .\]
   The very ampleness assumption ensures $V\to\wt{M\times M}$ is a projective bundle.
    \end{proof}

  This is used in the following key proposition: 
   
     \begin{proposition}[Voisin \cite{V1}]\label{voisin1} Assumptions as in lemma \ref{projbundle}. Assume moreover $M$ has trivial Chow groups. Let $R\in A^n(V)_{}$. Suppose that for all $b\in B$ one has
    \[ H^n(X_b)_{prim}\not=0\ \ \ \ 
  \hbox{and}\ \ \ \ 
     R\vert_{\wt{X_b\times X_b}}=0\ \ \in H^{2n}(\wt{X_b\times X_b})\ .\]
   Then there exists $\gamma\in A^n(M\times M)_{}$ such that
    \[     (p_b)_\ast \bigl(R\vert_{\wt{X_b\times X_b}}\bigr)= \gamma\vert_{X_b\times X_b}  \ \ \in A^{n}({X_b\times X_b})_{}\]  
    for all $b\in B$. 
   (Here $p_b$ denotes the restriction of $p$ to $\wt{X_b\times X_b}$, which is the blow--up of $X_b\times X_b$ along the diagonal.)
    \end{proposition}

\begin{proof} This is \cite[Proposition 1.6]{V1}.
\end{proof}

 The following is an equivariant version of proposition \ref{voisin1}:
 
  \begin{proposition}[Voisin \cite{V1}]\label{voisin2} Let $M$ and $L$ be as in proposition \ref{voisin1}. Let $G\subset\aut(M)$ be a finite group. Assume the following:
  
  \noindent
  (\rom1) The linear system $\vert L\vert^G:=\PP\bigl( H^0(M,L)^G\bigr)$ has no base--points, and the locus of points in $\wt{M\times M}$ parametrizing triples $(x,y,z)$ such that the length $2$ subscheme $z$ imposes only one condition on $\vert L\vert^G$ is contained in the union of (proper transforms of) graphs of non--trivial elements of $G$, plus some loci of codimension $>n+1$.
  
  \noindent
  (\rom2) Let $B\subset\vert L\vert^G$ be the open parametrizing smooth hypersurfaces, and let $X_b\subset M$ be a hypersurface for $b\in B$ general. There is no non--trivial relation
   \[ {\displaystyle\sum_{g\in G}} c_g \Gamma_g +\gamma=0\ \ \ \hbox{in}\ H^{2n}(X_b\times X_b)\ ,\]
   where $\gamma$ is a cycle in $\ima\bigl( A^n(M\times M)\to A^n(X_b\times X_b)\bigr)$.
   
   Let $R\in A^n(\XX\times_B \XX)$ be such that
     \[  R\vert_{{X_b\times X_b}}=0\ \ \in H^{2n}({X_b\times X_b})\ \ \ \forall b\in B\ .\]
    Then there exists $\gamma\in A^n(M\times M)_{}$ such that
    \[     R\vert_{{X_b\times X_b}}= \gamma\vert_{X_b\times X_b}  \ \ \in A^{n}({X_b\times X_b})\ \ \ \forall b\in B\ .\]  
  \end{proposition} 
  
 \begin{proof} This is not stated verbatim in \cite{V1}, but it is contained in the proof of \cite[Proposition 3.1 and Theorem 3.3]{V1}. We briefly review the argument.
 One considers
   \[  V:=\Bigl\{ \bigl((x,y,z),\sigma\bigr) \ \vert\ \sigma\vert_z=0\Bigr\}\ \ \subset\ \wt{M\times M}\times \vert L\vert^G\ .\]   
   The problem is that this is no longer a projective bundle over $\wt{M\times M}$. However, as explained in the proof of \cite[Theorem 3.3]{V1}, hypothesis (\rom1) ensures that one 
   can obtain a projective bundle after blowing up the graphs $\Gamma_g, g\in G$ plus some loci of codimension $>n+1$. Let $M^\prime\to\wt{M\times M}$ denote the result of these blow--ups, and let $V^\prime\to M^\prime$ denote the projective bundle obtained by base--changing. 

Analyzing the situation as in \cite[Proof of Theorem 3.3]{V1}, one obtains
   \[ R\vert_{X_b\times X_b} =R_0\vert_{X_b\times X_b}+ {\displaystyle\sum_{g\in G}} \lambda_g \Gamma_g\ \ \ \hbox{in}\ A^n(X_b\times X_b) \ ,\]
   where $R_0\in A^n(M\times M)$ and $\lambda_g\in\QQ$ (this is \cite[Equation (15)]{V1}).
   By assumption, $R\vert_{X_b\times X_b}$ is homologically trivial. Using hypothesis (\rom2), this implies that all $\lambda_g$ have to be $0$.   
     \end{proof}

\section{Main result}

\begin{theorem}\label{main} Let $X\subset\PP^5(\C)$ be a smooth cubic fourfold defined by an equation
  \[  f(x_0,x_1,x_2,x_3) + (x_4)^2\ell_1(x_0,\ldots,x_3) + (x_5)^2\ell_2(x_0,\ldots,x_3)+x_4x_5\ell_3(x_0,\ldots,x_3)=0\ \]
  (here $f$ has degree $3$ and $\ell_1, \ell_2,\ell_3$ are linear forms). There exists a $K3$ surface $S$ and a correspondence $\Gamma\in A^{3}(X\times S)$ inducing a split injection
   \[ \Gamma_\ast\colon\ \ \ A^3_{hom}(X)\ \hookrightarrow\ A^2_{hom}(S)\ .\]
\end{theorem}

\begin{proof} Let us consider the involution 
  \[  \begin{split} \iota\colon  \PP^5 \ &\to\ \PP^5\ ,\\
                         [x_0:x_1:x_2:x_3:x_4:x_5]\ &\mapsto\  [x_0:x_1:x_2:x_3: -x_4: -x_5]\ .\\
                      \end{split}\]
  The family of cubic fourfolds $X$ as in theorem \ref{main} is exactly the family of smooth cubic fourfolds invariant under $\iota$ (this was observed in \cite[Section 7]{Cam}, and also in \cite{LFu3}, where this family appears as ``Family V-(1)'' in the classification table of \cite[Theorem 0.1]{LFu3}).    
 Let us denote by
 \[ \iota_X\colon\ \ X\ \to\ X \]
 the involution of $X$ induced by $\iota$.
 
 Let $F:=F(X)$ denote the Fano variety parametrizing lines contained in $X$. The variety $F$ is a hyperk\"ahler variety \cite{BD}. The involution
  \[ \iota_F\colon\ \ F\ \to\ F \]
 induced by $\iota_X$ is symplectic \cite[Section 7]{Cam}, \cite[Theorem 0.1]{LFu3}. The fixed locus of $\iota_F$ consists of $28$ isolated points and a $K3$ surface $S\subset F$ \cite[Section 7]{Cam}, \cite[Section 4]{LFu3}. 
 The involution $\iota_F$ being symplectic, the surface $S\subset F$ is a {\em symplectic subvariety\/}, i.e. the inclusion $\tau\colon S\to F$ induces an isomorphism
   \[ \tau^\ast\colon\ \ H^{2,0}(F)\ \xrightarrow{\cong}\ H^{2,0}(S)\ .\]
As is readily seen, this implies there is also an isomorphism
 \begin{equation}\label{triso}  \tau^\ast\colon\ \ H^{2}_{tr}(F)\ \xrightarrow{\cong}\ H^{2}_{tr}(S)\ ,\end{equation}
 where $H^2_{tr}()\subset H^2()$ denotes the smallest Hodge--substructure containing $H^{2,0}()$. 
 Let $\Gamma_{BD}$ be the correspondence inducing the Beauville--Donagi isomorphism
  \begin{equation}\label{bdiso} (\Gamma_{BD})_\ast\colon\ \ H^4(X)\ \xrightarrow{\cong}\ H^2(F)\ \end{equation}
  \cite{BD}. (That is, let $P\subset X\times F$ denote the incidence variety, with morphisms $p\colon P\to F$, $q\colon P\to X$. Then $\Gamma_{BD}:= \Gamma_p\circ {}^t \Gamma_q\in A^{5}(X\times F)$.)
 
 Let us define a correspondence
  \[ \Gamma:=  {}^t \Gamma_\tau\circ \Gamma_{BD}\ \ \ \in A^{3}(X\times S)\ .\]
   Combining isomorphisms (\ref{triso}) and (\ref{bdiso}), we obtain an isomorphism
   \[ \Gamma_\ast\colon\ \ H^4(X)/N^2\ \xrightarrow{\cong}\ H^2_{tr}(F)\ \xrightarrow\ H^2_{tr}(S)\ .\]
   A bit more formally, this implies there is an isomorphism of homological motives
   \begin{equation}\label{motiso} \Gamma\colon\ \ (X,\pi^X_{4,1},0)\ \xrightarrow{\cong}\ (S,\pi^S_{2,tr},0)\ \ \ \hbox{in}\ \MM_{\rm hom}\ .\end{equation}
   Here, $\pi^X_{4,1}=\pi^X_4-\pi^X_{4,2}$ is a projector on $H^4(X)/N^2$; this exists thanks to theorem \ref{vial}. The projector $\pi^S_{2,tr}$ is the projector on $H^2_{tr}(S)$ constructed in \cite{KMP}. Let $\Psi\in A^{3}(S\times X)$ be a correspondence inducing an inverse to the isomorphism (\ref{motiso}). This means that we have
   \[ (\Psi\circ \Gamma)_\ast=\ide\colon\ \ \ H^4(X)/N^2\ \to\ H^4(X)/N^2\ ,\]
   which means that there is a homological equivalence of cycles
   \begin{equation}\label{homeq}  \Psi\circ\Gamma\circ \pi_4^X  =\pi_4^X+\gamma_1\ \ \ \hbox{in}\ H^8(X\times X)\ ,\end{equation}
   where $\gamma_1\in A^4(X\times X)$ is some cycle supported on $V\times V\subset X\times X$, where $V\subset X$ is a codimension $2$ closed subvariety (this is because $\gamma_1$ is supported on the support of $\pi^X_{4,2}$, which is supported on $V\times V$ as indicated, by theorem \ref{vial}).
   
   As $X\subset\PP^5$ is a hypersurface, the only interesting K\"unneth component is $\pi^X_4$. That is, we can write
   \[ \Delta_X =\pi^X_4 +\gamma_2 \ \ \ \hbox{in}\ H^8(X\times X)\ ,\]   
   where $\gamma_2$ is a ``completely decomposed'' cycle, i.e. a cycle with support on $\cup_i V_i\times W_i\subset X\times X$, where $\dim V_i +\dim W_i=4$. Plugging this in equation (\ref{homeq}), we obtain a homological equivalence of cycles
   \begin{equation}\label{homeq2}  \Psi\circ\Gamma =\Delta_X +\gamma \ \ \ \hbox{in}\ H^8(X\times X)\ ,\end{equation}
   where $\gamma$ is a ``completely decomposed'' cycle in the above sense.
   
   We now proceed to upgrade the homological equivalence (\ref{homeq2}) to a rational equivalence. This can be done thanks to the work of Voisin on the Bloch/Hodge equivalence \cite{V0}, \cite{V1}, using the technique of ``spread'' of algebraic cycles in good families.
   
 Following the approach of \cite{V0}, \cite{V1}, we put the above construction in family. We define
   \[ \pi\colon\ \ \XX\ \to\ B \]
  to be the family of all smooth cubic fourfolds given by an equation as in theorem \ref{main}. (That is, we let $G\subset\aut(\PP^5)$ be the order $2$ group generated by the involution $u$, and we define
    \[ B\ \subset\ \Bigl(\PP H^0\bigl(\PP^5,\OO_{\PP^5}(3)\bigr)\Bigr)^G \]
    as the open subset parametrizing smooth $G$--invariant cubics.)  We will write $X_b:=\pi^{-1}(b)$ for the fibre over $b\in B$. We also define families
    \[ \FF\ \to\ B\ ,\ \ \ \Ss\ \to\ B \]
   of Fano varieties of lines, resp. of $K3$ surfaces. (That is, $\Ss\subset\FF$ is the fixed locus of the involution of $\FF$ induced by $\iota$.) We will write $F_b$ and $S_b$ for the fibre over $b\in B$.

The correspondence $\Gamma$ constructed above readily extends to this relative setting: 
 
  \begin{lemma}\label{relgamma} There exists a relative correspondence $\Gamma\in A^3(\XX\times_B \Ss)$, such that for all $b\in B$, the restriction
    \[ \Gamma_b:= \Gamma\vert_{X_b\times S_b}\ \ \ \in A^3(X_b\times S_b) \]
    induces the isomorphism
    \[ \Gamma_b\colon\ \  (X_b,\pi^{X_b}_{4,1},0)\ \xrightarrow{\cong}\ (S_b,\pi^{S_b}_{2,tr},0)\ \ \ \hbox{in}\ \MM_{\rm hom}   \]
    as in (\ref{motiso}).
   \end{lemma} 
   
  \begin{proof} Let $\Pp\subset \XX\times_B \FF$ denote the incidence variety, with projections $p\colon \Pp\to \FF$, $q\colon \Pp\to \XX$. Let $\tau$ denote the inclusion morphism $\Ss\to \FF$. We define
   \[ \Gamma:= {}^t \Gamma_\tau\circ  \Gamma_p\circ {}^t \Gamma_q\in A^{3}(\XX\times_B \Ss)\ . \]
 (For composition of relative correspondences in the setting of smooth quasi--projective families that are smooth over a base $B$, cf.   
   \cite{CH}, \cite{GHM}, \cite{NS}, \cite{DM}, \cite[8.1.2]{MNP}.) 
    \end{proof}
   
 The correspondences $\Psi$ and $\gamma$ also extend to the relative setting:
 
 \begin{lemma}\label{allrel} There exist subvarieties $\VV_i, \WW_i\subset \XX$ with $\codim(\VV_i)+\codim(\WW_i)=4$, and
 relative correspondences
  \[ \Psi\ \ \in A^3(\Ss\times_B \XX)\ , \ \ \ \gamma\ \ \in A^4(\XX\times_B \XX)\ ,\]
  where $\gamma$ is supported on $\cup_i \VV_i\times_B \WW_i$, and such that for all $b\in B$, the restrictions
  \[ \Psi_b:= \Psi\vert_{S_b\times X_b}\ \ \in A^3(S_b\times X_b)\ ,\ \ \ \gamma_b:= \gamma\vert_{X_b\times X_b}\ \ \in A^4(X_b\times X_b) \]
  verify the equality
  \[  \Psi_b\circ\Gamma_b =\Delta_{X_b} +\gamma_b \ \ \ \hbox{in}\ H^8(X_b\times X_b) \]
  as in (\ref{homeq2}).
  \end{lemma}  
  
  \begin{proof} The statement is different, but this is really the same Hilbert schemes argument as \cite[Proposition 3.7]{V0}, \cite[Proposition 4.25]{Vo}. 
  
  Let $\Gamma\in A^3(\XX\times_B \Ss)$ be the relative correspondence of lemma \ref{relgamma}, and let $\Delta_\XX\in A^4(\XX\times_B \XX)$ be the relative diagonal. By what we have said above, for each $b\in B$ there exist subvarieties $V_{b,i}, W_{b,i}\subset X_b$ (with $\dim (V_{b,i})+\dim(W_{b,i})=4$), and a cycle $\gamma_b$ supported on 
    \[\cup_i V_{b,i}\times W_{B,i}\subset X_b\times X_b\ ,\] 
    and a cycle $\Psi_b\in A^3(S_b\times X_b)$, such that there is equality
  \begin{equation}\label{split}  \Psi_b\circ \Gamma_b = \Delta_\XX\vert_{X_b\times X_b} +\gamma_b\ \ \ \hbox{in}\ H^8(X_b\times X_b)\ .\end{equation}
  
  The point is that the data of 
   all the
   $(b,V_{b,i}, W_{b,i}, \gamma_b,\Psi_b)$ that are solutions of the equality (\ref{split})
   can be encoded by a countable number of algebraic varieties $p_j\colon M_j\to B$, with universal objects 
   \[  \VV_{i,j}\to M_j\ , \ \ \ \WW_{i,j}\to M_j\ ,\ \ \  \gamma_j\to M_j\ ,\ \ \ \Psi_j\to M_j\ \]
   (where $\VV_{i,j}, \WW_{i,j}\subset \XX_{M_j}$, and $\gamma_j$ is a cycle supported on $\cup_i \VV_{i,j}\times_{M_j}\WW_{i,j}$, and $\Psi_j\in A^3(\Ss\times_{M_j} \XX)$),
       with the property that 
   for $m\in M_j$ and $b=p_j(m)\in B$, we have
     \[  \begin{split}   &(\gamma_j)\vert_{X_b\times X_b}=  \gamma_b\ \ \ \hbox{in}\ H^8(X_b\times X_b)\ ,\\
                                & (\Psi_j)\vert_{S_b\times X_b}= \Psi_b\ \ \ \hbox{in}\ H^6(S_b\times X_b)\ .\\
                                \end{split}\]
     By what we have said above, the union of the $M_j$ dominate $B$. Since there is a countable number of $M_j$, one of the $M_j$ (say $M_0$) must dominate $B$. Taking hyperplane sections, we may assume $M_0\to B$ is generically finite (say of degree $d$). Projecting the cycles $\gamma_0$ and $\Psi_0$  to $\XX\times_B \XX$, resp. to $\Ss\times_B \XX$, and then dividing by $d$, we have obtained cycles $\gamma$ and $\Psi$ as requested.
  \end{proof}
   
  Lemma \ref{allrel} can be succinctly restated as follows: the relative correspondence
    \[ R:= \Psi\circ \Gamma - \Delta_{\XX} - \gamma\ \ \ \in A^4(\XX\times_B \XX) \]
    has the property that for all $b\in B$, the restriction is homologically trivial:
    \[ R\vert_{X_b\times X_b}\ \ \ \in A^4_{hom}(X_b\times X_b)\ \ \ \forall b\in B\ .\]
    
    Applying theorem \ref{voisin2} to $R$ (this is possible in view of proposition \ref{OK} below), we find that 
    \begin{equation}\label{rateq} (R +\delta)\vert_{X_b\times X_b}=0\ \ \ \hbox{in}\ A^4(X_b\times X_b)\ \ \ \forall b\in B\ ,\end{equation}
    where $\delta$ is some cycle  
       \[  \delta\in \ima \Bigl(  A^4(\PP^5\times \PP^5\times B)\ \to\ A^4(\XX\times_B \XX)\Bigr)\ .\]
     Since $A^\ast_{hom}(\PP^5\times\PP^5)=0$, we have
     \[ (\delta\vert_{X_b\times X_b})_\ast A^\ast_{hom}(X_b)=0\ .\]
     For $b\in B$ general, the fibre $X_b\times X_b$ will be in general position with respect to the $\VV_i$ and $\WW_i$ and so
     \[ \dim(\VV_i\cap X_b)+\dim(\WW_i\cap X_b) =4\ \ \forall i\ ,\]
     which ensures that
     \[ (\gamma\vert_{X_b\times X_b})_\ast A^\ast_{hom}(X_b)=0\ .\]
    Plugging in the definition of $R$ into the rational equivalence (\ref{rateq}), this means that
     \[ (\Psi\vert_{X_b\times X_b})_\ast (\Gamma\vert_{X_b\times X_b})_\ast =\ide\colon\ \ A^\ast_{hom}(X_b)\ \to\ A^\ast_{hom}(X_b)\ \ \ \hbox{for}\ b\in B\ \hbox{general}\ ,\]
     which proves theorem \ref{main} for $b\in B$ general.
     
     To prove theorem \ref{main} for any given $b_0\in B$, we note that the above construction can also be made locally around the point $b_0$: in the construction of lemma \ref{allrel}, we throw away all the data $M_j$ for which the subvarieties $\VV_{i,j}, \WW_{i,j}$ are {\em not\/} all in general position with respect to $X_{b_0}\times X_{b_0}$. The union of the remaining $M_j$ will dominate an open $B^\prime\subset B$ containing $b_0$, and so the above proof works for the cubic $X_{b_0}$.     
    
 To end the proof, it remains to verify the hypotheses of theorem \ref{voisin2} (which we applied above) are met with. This is the content of the following:
 
 \begin{proposition}\label{OK} Let $\XX\to B$ be the family of smooth cubic fourfolds as in theorem \ref{main}, i.e.
    \[ B\ \subset\ \Bigl(\PP H^0\bigl(\PP^5,\OO_{\PP^5}(3)\bigr)\Bigr)^G \]
    is the open subset parametrizing smooth $G$--invariant cubics, and $G=\{id,\iota\}\subset\aut(\PP^5)$ as above. This set--up verifies the hypotheses of proposition \ref{voisin2}.
    \end{proposition}
  
  \begin{proof} Let us first prove hypothesis (\rom1) of proposition \ref{voisin2} is satisfied. 
  
  To this end, we consider the tower of morphisms
   \[ p\colon\ \ \PP^5\ \xrightarrow{p_1}\ P^\prime:= \PP^5/G\ \xrightarrow{p_2}\ P:=\PP(1^4,2^2)\ ,\]
   where $\PP(1^4,2^2)=\PP^5/(\ZZ/2\ZZ\times \ZZ/2\ZZ)$ denotes a weighted projective space. Let us write $\iota_4, \iota_5$ for the involutions of $\PP^5$ 
   \[  \begin{split}  &\iota_4 [x_0:\ldots:x_5] = [x_0:x_1:x_2:x_3:-x_4:x_5]\ ,\\
                            &\iota_5 [x_0:\ldots:x_5] = [x_0:x_1:x_2:x_3:x_4:-x_5]\ .\\
                        \end{split} \]    
   (We note that $\iota=\iota_4\circ\iota_5$, and $P=\PP^5/ <\iota_4,\iota_5>$.)
   
   The sections in $\bigl(\PP H^0\bigl(\PP^5,\OO_{\PP^5}(3)\bigr)\bigr)^G$ are in bijection with $ \PP H^0\bigl(P^\prime,\OO_{P^\prime}(3)\bigr) $, and so there is an inclusion
   \[  \Bigl(\PP H^0\bigl(\PP^5,\OO_{\PP^5}(3)\bigr)\Bigr)^G \ \subset\ p^\ast \PP H^0\bigl( P,\OO_P(3)\bigr)\ .\]  
   
   Let us now assume $x,y\in\PP^5$ are two points such that
   \[ (x,y)\not\in \Delta_{\PP^5}\cup \Gamma_{\iota_4}\cup \Gamma_{\iota_5}\cup \Gamma_\iota\ .\]
   Then 
   \[ p(x)\not=p(y)\ \ \ \hbox{in}\ P\ ,\]
   and so (using lemma \ref{delorme} below) there exists $\sigma\in\PP H^0\bigl(P^\prime,\OO_{P^\prime}(3)\bigr)$ containing $p(x)$ but not $p(y)$. The pullback $p^\ast(\sigma)$ contains $x$ but not $y$, and so these points $(x,y)$ impose $2$ independent conditions on   $ \bigl(\PP H^0\bigl(\PP^5,\OO_{\PP^5}(3)\bigr)\bigr)^G$.
   
   It only remains to check that a generic element $(x,y)\in\Gamma_{\iota_4}\cup\Gamma_{\iota_5}$ also imposes $2$ independent conditions. Let us assume $(x,y)$ is generic on $\Gamma_4$ (the argument for $\Gamma_5$ is only notationally different). Let us write $x=[a_0:a_1:\ldots:a_5]$. By genericity, we may assume all $a_i$ are $\not=0$ (intersections of $\Gamma_4$ with a coordinate hyperplane have codimension $>n+1$ and so need not be considered for hypothesis (\rom1) of proposition \ref{voisin2}). We can thus write
   \[ x=[1:a_1:a_2:a_3:a_4:a_5]\ ,\ \ \ y= [1:a_1:a_2:a_3:-a_4:a_5] \ ,\ \ \ a_i\not=0 \ .\] 
   The cubic 
   \[  a_5 x_0(x_4)^2 - a_4x_0x_4x_5=0 \]
   is $G$--invariant and contains $x$ while avoiding $y$. This proves hypothesis (\rom1) is satisfied.
   
  To establish hypothesis (\rom2) of proposition \ref{voisin2}, we proceed by contradiction. Let us suppose hypothesis (\rom2) is not met with, i.e. there exists a smooth cubic $X_b$ as in theorem \ref{main}, and a non--trivial relation
  \[  c\,\Delta_{X_b} +d\, \Gamma_{\iota_{X_b}} +\delta =0\ \ \ \hbox{in}\ H^8(X_b\times X_b)\ ,\]
  where $c,d\in \QQ^\ast$ and $\delta\in\ima\bigl( A^4(\PP^5\times\PP^5)\to A^4(X_b\times X_b)\bigr)$.
  Looking at the action on $H^{3,1}(X_b)$, we find that necessarily $c=-d$ (indeed, $\delta$ does not act on $H^{3,1}(X_b)$, and $\iota$ acts as the identity on $H^{3,1}(X_b)$).  
   That is, we would have a relation
   \[ \Delta_{X_b} -\Gamma_{\iota_{X_b}} +{1\over c}\ \delta =0\ \ \ \hbox{in}\ H^8(X_b\times X_b)\ .\]
   Looking at the action on $H^{2,2}(X_b)$, we find that
    \[ (\iota_{X_b})^\ast =\ide\colon\ \ \ \gr^2_F H^4(X_b,\C)_{\rm prim}\ \to\   \gr^2_F H^4(X_b,\C)_{\rm prim}\ .\]
    Since there is a codimension $2$ linear subspace in $\PP^5$ fixed by $\iota$, it follows that actually
    \[ (\iota_{X_b})^\ast =\ide\colon\ \ \ \gr^2_F H^4(X_b,\C)_{}\ \to\   \gr^2_F H^4(X_b,\C)_{}\ .\]
    Consider now the Fano variety of lines $F=F(X_b)$ with the involution $\iota_F$. Using the Beauville--Donagi isomorphism \cite{BD}, one obtains that also
    \[ (\iota_{F})^\ast =\ide\colon\ \ \ \gr^1_F H^2(F,\C)_{}\ \to\   \gr^1_F H^2(F,\C)_{}\ .\]

    As $\dim  \gr^1_F H^2(F,\C)_{}=21$, this would imply that the trace of $(\iota_{F})^\ast$ on $ \gr^1_F H^2(F,\C)_{}$ is $21$.
    However, this contradicts proposition \ref{chiara} below, and so hypothesis (\rom2) must be satisfied.  
  
  \begin{lemma}\label{delorme} Let $P=\PP(1^4,2^2)$. Let $r,s\in P$ and $r\not=s$. Then there exists $\sigma\in\PP H^0\bigl(P,\OO_P(3)\bigr)$ containing $r$ but avoiding 
  $s$.
  \end{lemma}
  
  \begin{proof} It follows from Delorme's work \cite[Proposition 2.3(\rom3)]{Del} that the locally free sheaf $\OO_P(2)$ is very ample. This means there exists $\sigma^\prime\in\PP H^0\bigl(P,\OO_P(2)\bigr)$ containing $r$ but avoiding 
  $s$. Taking the union of $\sigma^\prime$ with a hyperplane avoiding $s$, one obtains $\sigma$ as required.
    \end{proof}
  
\begin{proposition}[Camere \cite{Cam}]\label{chiara} Let $X_b\subset\PP^5$ be a cubic as in theorem \ref{main}, and let $\iota_{X_b}$ be the involution as above. Let $F=F(X_b)$ be the Fano variety of lines, and let $\iota_F$ be the involution of $F$ induced by $\iota_{X_b}$. The trace of $(\iota_{F})^\ast$ on the
$21$--dimensional vector space $\gr^1_F H^2(F,\C)$ is $5$.
\end{proposition}  

\begin{proof} This follows from \cite[Theorem 5]{Cam}.
  \end{proof}

   \end{proof}
   \end{proof}

\begin{remark} Let $X$ and $S$ be as in theorem \ref{main}. One expects there is actually an isomorphism
  \[ \Gamma_\ast\colon\ \ \ A^3_{hom}(X)\ \xrightarrow{\cong}\ A^2_{hom}(S)\ .\]
  I am unsure whether the argument of theorem \ref{main} can also be used to prove surjectivity.
  \end{remark}

\begin{remark} To find the $K3$ surface $S$ of theorem \ref{main}, we have used the existence of the symplectic involution $\iota_F$ on the Fano variety $F=F(X)$ of lines on the cubic fourfold $X$, for which $S\subset F$ is in the fixed locus. One could ask if there exist cubic fourfolds $X$ other than those of theorem \ref{main}, such that the Fano variety $F(X)$ has a symplectic automorphism with a $2$--dimensional component in the fixed locus.

However, if one restricts to {\em polarized\/} symplectic automorphisms of $F(X)$, there are only $2$ families with a surface in the fixed locus: the family of theorem \ref{main}, and a family with an abelian surface in the fixed locus. This follows from the classification obtained by L. Fu in \cite[Theorem 0.1]{LFu3} (the first family is labelled ``Family V-(1)'', and the second family is labelled ``Family IV-(2)'' in loc. cit.).

The second family (with an abelian surface in the fixed locus) is studied from the point of view of algebraic cycles in \cite{voisinHK}.
\end{remark}

\begin{remark} Let $X$ and $F$ be as in theorem \ref{main}. We mention in passing that the automorphisms $\iota$ and $\iota_F$ of $X$ resp. of $F$ act as the identity on 
$A^3(X)$, resp. on $A^4(F)$ (for $X$, this follows immediately from theorem \ref{main}). 

This is proven more generally for any polarized symplectic automorphism of the Fano variety of lines of a cubic fourfold \cite[Theorems 0.5 and 0.6]{LFu2} (for a slightly different take on this, cf. \cite[Theorem 5.3]{SV}). The argument of \cite{LFu2} is (just like the argument proving theorem \ref{main}) based on the idea of spread of algebraic cycles in a family, inspired by \cite{V0}, \cite{V1}.
\end{remark}



\section{Finite--dimensionality}

\begin{corollary}\label{findim} Let $X\subset\PP^5(\C)$ be a smooth cubic fourfold defined by an equation
  \[  f(x_0,x_1,x_2,x_3) + (x_4)^2\ell_1(x_0,\ldots,x_3) + (x_5)^2\ell_2(x_0,\ldots,x_3)+x_4x_5\ell_3(x_0,\ldots,x_3)=0\ .\]
  Assume
   \[ \dim H^4(X)\cap H^{2,2}(X,\C)\ge 20 \ .\]
   Then $X$ has finite--dimensional motive.
   \end{corollary}
   
 \begin{proof} It follows from (the proof of) theorem \ref{main} there is an inclusion as direct summand
   \begin{equation}\label{inclu} h(X)\ \subset\ h(S)(1)\oplus\bigoplus_j \LLL(m_j)\ \ \ \hbox{in}\ \MM_{\rm rat}\ ,\end{equation}
   where $S$ is a $K3$ surface. We have also seen (in the proof of theorem \ref{main}) there is an isomorphism
   \[ \Gamma_\ast\colon\ \ \ H^4(X)/N^2\ \xrightarrow{\cong}\ H^2_{tr}(S)\ .\]
   Since the Hodge conjecture is known for $X$ (because $X$ is Fano), there is equality
   \[ N^2 H^4(X)=H^4(X)\cap H^{2,2}(X,\C)\ .\]
   Thus, the hypothesis on the dimension of the space of Hodge classes implies that 
     \[ \dim N^2 H^4(X)\ge 20\ ,\] 
     and so
   \[ \dim H^2_{tr}(S) = \dim (  H^4(X)/N^2 ) = 23-\dim N^2\le 3\ .\]
   This implies the Picard number $\rho(S)$ is at least $19$, and so $S$ has finite--dimensional motive \cite{Ped}. In view of inclusion (\ref{inclu}), this concludes the proof.
   \end{proof}

 \begin{remark}\label{fanotoo} Let $X$ be a cubic as in corollary \ref{findim}. Applying \cite{fano}, it follows that the Fano variety of lines $F:=F(X)$ also has finite--dimensional motive.
 \end{remark}

\vskip1cm
\begin{nonumberingt} Thanks to all participants of the Strasbourg 2014/2015 ``groupe de travail'' based on the monograph \cite{Vo} for a very pleasant atmosphere.
Many thanks to Kai and Len and Yoyo for stimulating discussions not related to this work.
\end{nonumberingt}

\end{document}